\documentclass[11pt,a4paper]{amsart}
\usepackage[T2A]{fontenc}
\usepackage[cp1251]{inputenc}

\usepackage{latexsym,amssymb,amscd,mathrsfs,bbm}
\usepackage[all]{xy}
\usepackage[ukrainian,english]{babel}

 \usepackage[notref,notcite]{showkeys}

\newtheorem{theorem}{Theorem}[section]
\newtheorem{prop}[theorem]{Proposition}
\newtheorem{corol}[theorem]{Corollary}

\theoremstyle{definition}
\newtheorem{defin}[theorem]{Definition}

\newtheorem{exam}[theorem]{Example}
\theoremstyle{remark}

\numberwithin{equation}{section}

\def\iff{if and only if }

\def\*{\otimes}
\def\+{\oplus}
\def\xx{\times}
\def\ol{\overline}
\def\bl{\textup{(}}	\def\br{\textup{)}}
\def\xarr{\xrightarrow}
\def\8{\infty}

\def\hom{\mathop\mathrm{Hom}\nolimits}

\def\hos{\mathop\mathrm{Hos}\nolimits}
\def\hot{\mathop\mathrm{Hot}\nolimits}

\def\ob{\mathop\mathrm{Ob}\nolimits}
\def\im{\mathop\mathrm{Im}\nolimits}
\def\tors{\mathop\mathrm{tors}\nolimits}

\def\rE{\mathrm E}		\def\rF{\mathrm F}

\def\mZ{\mathbb Z}
\def\mN{\mathbb N}
\def\mQ{\mathbb Q}

\def\ga{\gamma}
\def\la{\lambda}
\def\Ga{\Gamma}
\def\al{\alpha}		
\def\be{\beta}		

\def\De{\Delta} 
\def\la{\lambda}	
\def\vi{\varphi}
\def\La{\Lambda}

\def\cA{\mathscr A}		\def\cC{\mathscr C}	
\def\cB{\mathscr B}		\def\cM{\mathscr M}
\def\cS{\mathscr S}		
\def\cE{\mathscr E}		\def\cI{\mathscr I}		
\def\cJ{\mathscr J}		

\def\kM{\mathcal M}		\def\kN{\mathcal N}	
\def\dE{\mathfrak E}	\def\dF{\mathfrak F}
\def\dX{\mathfrak X}	

\def\bC{\mathbf C}		\def\bA{\mathbf A}
\def\bE{\mathbf E}		\def\bB{\mathbf B}

\def\mno{\mN\cup\{0\}}
\def\mni{\mN\cup\{\8\}}
\def\oS{\ol\cS}

\def\Ab{\mathsf{Ab}}	\def\op{^\mathrm{op}}

\def\set#1{\left\{\,#1\,\right\}}
\def\setsuch#1#2{\left\{\,#1\mid #2\,\right\}}
\def\lst#1#2{ #1_1 , #1_2 , \dots , #1_{#2} }
\def\mtr#1{\begin{pmatrix}#1\end{pmatrix}}

\def\cel{*=0{\bullet}}

\def\sb{\subset}	
	\def\spe{\supseteq}
	\def\df{\mbox{-}}
\def\bop{\bigoplus}

\def\arb{\ar@{<-}}	\def\oo{{\ 0\ }}

\begin{document}

\title[$p$-localization of stable homotopy category]%
{Atoms in the $p$-localization of stable homotopy category}
\author[Y. Drozd]{Yuriy A. Drozd} 
\author[P. Kolesnyk]{Petro O. Kolesnyk}
\address{Institute of Mathematics, National Academy of Sciences of Ukraine,
01601 Kyiv, Ukraine}
\email{y.a.drozd@gmail.com, drozd@imath.kiev.ua}
\urladdr{www.imath.kiev.ua/$\sim$drozd}
\email{iskorosk@gmail.com}

\begin{abstract}
We study $p$-localizations, where $p$ is an odd prime, of the full subcategories $\cS^n$ of stable 
homotopy category consisting of CW-complexes having cells in $n$ successive dimensions. 
Using the technique of triangulated categories and matrix problems we classify
atoms (indecomposable objects) in $\cS_p^n$ for $n\le4(p-1)$ and show that for $n>4(p-1)$
such classification is wild in the sense of the representation theory.
\end{abstract}

\maketitle
\tableofcontents 

\section*{Introduction}

 Classification of homotopy types of polyhedra (finite CW-complexes) is an old problem. It is well-known
 that it becomes essentially simpler if we consider the stable situation, i.e. identify two polyhedra
 having homotopy equivalent (iterated) suspensions. It leads to the notion of stable homotopy category
 and stable homotopy equivalence. Such a classification has been made for polyhedra of low dimensions
 by several authors; a good survey of these results is the paper of Baues \cite{ba}. Unfortunately, it
 cannot be done for higher dimensions, since the problem becomes extremely complicated. Actually, it
 results in ``\emph{wild problems}'' of the representation theory, i.e. problems containing 
 classification of representations of all finitely generated algebras over a field (cf. \cite{bd,d1,d2}; 
 for generalities about wild problems see the survey \cite{d0}). 
 
 In the survey \cite{d1} the first author proposed a new approach to the stable homotopy classification
 which seems more ``algebraic'' and simpler for calculations. It is based on the triangulated structure
 of the stable homotopy category and uses the technique of ``matrix problems'', more exactly, 
 \emph{bimodule categories} in the sense of \cite{d0}. In particular, it gave simplified proofs of the 
 results of \cite{bh,bd,bdf}. In \cite{d2} this technique gave new results on classification of polyhedra
 with torsion free homologies.
 
 The main difficulties in the stable homotopy classification are related to the $2$-components of
 homotopy groups. That is why it is natural to study \emph{$p$-local polyhedra}, 
 where $p$ is an odd prime;
 then we only use the $p$-parts of homotopy groups. In this paper we use the technique of
 \cite{d1,d2} to classify $p$-local polyhedra that only have cells in $n$ successive dimensions for
 $n\le4(p-1)$. Analogous results have been obtained by Henn \cite{henn}, who used a different
 approach. Our description seems more straightforward and more visual. It gives explicit construction
 of polyhedra by successive attaching simpler polyhedra to each other. We also show that for $n>4(p-1)$
 the stable classification of $p$-local polyhedra becomes a wild problem, so the obtained results
 are in some sense closing.
 
 Section 1 covers the main notions from the stable homotopy theory, bimodule categories and their
 relations. In Section 2 we calculate morphisms between Moore polyhedra and their products.
 In Section 3 we describe polyhedra in the case $n=2p-1$. This classification happens to be 
 ``\emph{essentially finite}'' in the sense that there is an upper bound for the number of cells
 in indecomposable polyhedra (\emph{atoms}); actually, atoms have at most 4 cells. 
 Section 4 is the main one. Here we describe polyhedra for $2p\le n\le4(p-1)$. The result is
 presented in terms of \emph{strings and bands}, which is usual in the modern representation 
 theory. String and band polyhedra are defined by some combinatorial invariant (a \emph{word})
 and, in band case, an irreducible polynomial over the residue field $\mZ/p$. In the
 representation theory such description is said to be \emph{tame}. Finally, in Section 5 we prove
 that the classification becomes wild if $n>4(p-1)$. 
 
 The description obtained by matrix methods is
 \emph{local}, just as that of \cite{henn}. Using the results of \cite{dk} we also obtain a global
 description of $p$-primary polyhedra. Fortunately, it almost coincides with the local one, except rare 
 special cases, when one local object gives rise to $(p-1)/2$ global ones.
  
 The first author expresses his thanks to H.-J.~Baues, who introduced him into the world of algebraic
 topology and was his co-author in several first papers on this topic.
 
\section{Stable homotopy category and bimodule categories}
\label{s1} 

 We use basic definitions and facts concerning stable homotopy from \cite{co}. We denote by 
 $\cS$ the stable homotopy category of \emph{polyhedra}, i.e. finite CW-complexes. It is an additive
 category and the morphism groups in it are $\hos(X,Y)=\varinjlim_k\hot(X[k],Y[k])$, where $X[k]$ denotes
 the $k$-fold suspension of $X$ and $\hot(X,Y)$ denotes the set of homotopy classes of continuous
 maps $X\to Y$. Note that the direct sum in this category is the wedge (bouquet, or one-point gluing) 
 $X\vee Y$ and the natural map $\hos(X,Y)\to\hos(X[k],Y[k])$ is an isomorphism. 
 In what follows, we always deal with polyhedra as the objects of this category.
 In particular, \emph{isomorphism} means stable homotopy equivalence. Note that all groups
 $\hos(X,Y)$ are finitely generated and the stable homotopy groups $\pi^S_n(X)=\hos(S^n,X)$ are
 torsion if $n>\dim X$. 
 It is convenient to formally add to $\cS$ the ``negative shifts'' $X[-k]$ ($k\in\mN$)
 of polyhedra with the natural sets of morphisms, so that $X[k][l]\simeq X[k+l]$ and
 $\hos(X[k],Y[k])\simeq\hos(X,Y)$ for all $k\in\mZ$. Then $\cS$ becomes a \emph{triangulated category}, 
 where the suspension plays role of the shift and the exact triangles are cofibre sequences 
 $X\to Y\to Z\to X[1]$ (in $\cS$ they are the same as fibre sequences). From now on we consider
 $\cS$ with these additional objects. Actually, the category obtained in this way is equivalent to
 the category of finite $S$-spectra \cite{co,sw}.
   
 We denote by $\cS^n$ the full subcategory of $\cS$ whose objects are the shifts $X[k]$ ($k\in\mZ$)
 of polyhedra only having cells in at most $n$ successive dimensions, or, the same, $(m-1)$-connected 
 and of dimension at most $n+m$ for some $m$. The Freudenthal Theorem \cite[Theorem 1.21]{co} implies 
 that every object of $\cS^n$ is a shift (iterated suspension) of an $n$-connected polyhedron of 
 dimension at most $2n-1$. We denote the full subcategory of $\cS^n$ consisting of such polyhedra by 
 $\oS^n$. Moreover, if two such polyhedra are isomorphic in $\cS$, they are homotopy equivalent. 
 Following Baues \cite{ba}, we call an object from $\cS^n$ an \emph{atom} if it belongs to $\oS^n$, 
 does not belong to $\cS^{n-1}$ and is indecomposable (into a wedge of non-contractible polyhedra).
 
 Recall that the \emph{$p$-localization} of an additive category $\cC$ is the category $\cC_p$ such that
 $\ob\cC_p=\ob\cC$ and $\hom_{\cC_p}(A,B)=\mZ_p\*\hom_\cC(A,B)$, where $\mZ_p\sb\mQ$ is the subring
 $\setsuch{\frac ab}{a,b\in\mZ,\,p\nmid b}$. We consider the localized categories $\cS_p$ and $\cS^n_p$
 and denote their groups of morphisms $X\to Y$ by $\hos_p(X,Y)$. Actually, $\cS_p$ coincides with the 
 stable homotopy category of finite \emph{$p$-local CW-complexes} in the sense of \cite{su}. Every such
 space can be considered an image in $\cS_p$ of a \emph{$p$-primary polyhedron}, i.e. such polyhedron $X$
 that the map $p^k1_X$ for some $k$ can be factored through a wedge of spheres \cite{co}.
 
 To study the categories $\cS_p^n$ we use the technique of \emph{bimodule categories}, like in 
 \cite{d2}. We recall the corresponding notions. 
 
 \begin{defin}[cf. {\cite[Section 4]{d0}}]\label{11}
 Let $\cA$ and $\cB$ be additive categories, $\cM$ be an $\cA\df\cB$-bimodule, i.e. a biadditive functor
 $\cA\op\xx\cB\to\Ab$ (the category of abelian groups). The \emph{bimodule category} $\cE(\cM)$ (or the
 category of elements of $\cM$) is defined as follows.
 \begin{itemize}
\item
 $\ob\cE(\cM)=\bigcup_{\substack{A\in\ob\cA\\B\in\ob\cB}}\cM(A,B)$.
 
 \smallskip
\item
If $u\in\cM(A,B),\,v\in\cM(A',B')$, then 
 $$\hom_{\cE(\cM)}(u,v)=\setsuch{(f,g)}{f:A'\to A,\,g:B\to B',\,gu=fv}.$$
 \end{itemize}
 $\cE(\cM)$ is also an additive category. Note that we only consider \emph{bipartite bimodules} in the
 sense of \cite{d0}.
 \end{defin}
 
 Usually we choose a set of \emph{additive generators} of $\cA$ and $\cB$, i.e. sets
 $\set{\lst As}\sb\ob\cA$ and $\set{\lst Br}\sb\ob\cB$ such that every object from $\cA$ (respectively,
 from $\cB$) is isomorphic to a direct sum $\bop_{j=1}^sk_jA_j$ (respectively, $\bop_{i=1}^rl_iB_i$).
 Then an object of $\cE(\cM)$ can be presented as a block matrix $F=(F_{ij})$, where $F_{ij}$ is a matrix
 of size $l_i\xx k_j$ with coefficients from $\cM(A_j,B_i)$. If we present morphisms in the analogous
 matrix form, the action of morphisms on elements from $\cM$ is presented by the usual matrix
 multiplication. 
 
 We use the following localized version of \cite[Theorem 2.2]{d2}.
 
 \begin{theorem}\label{12} 
 Let $n\le m<2n-1$. Denote by $\cA$ \bl respectively, by $\cB$\br\ the full subcategory of
 $\cS_p$ consisting of $(m-1)$-connected polyhedra of dimension at most $2n-2$ \bl respectively, of
 $(n-1)$-connected polyhedra of dimension at most $m$\br. Consider the $\cA\df\cB$-bimodule
 $\cM$ such that $\cM(A,B)=\hos_p(A,B)$. Let $\cI$ be the ideal of the category $\cE(\cM)$
 consisting of all morphisms $(\al,\be):f\to f'$ such that $\al$ factors through $f$ and $\be$ factors 
 through $f'$. Let also $\cJ$ be the ideal of $\oS_p^n$ consisting of all maps $f:X\to Y$ such that $f$
 factors both through an object from $\cA[1]$ and through an object from $\cB$. The map
 $f\mapsto Cf$ \bl the \emph{cone} of $f$\br\ induces an equivalence 
 $\cE(\cM)/\cI\simeq\oS_p^n/\cJ$. 
 Moreover, $\cJ^2=0$, hence the isomorphism classes of the categories $\oS_p^n$ and $\oS_p^n/\cJ$ are 
 the same. 
 
 \emph{Note also that all groups $\cJ(X,Y)$ are finite \cite[Corollary 1.10]{dk}}.
 \end{theorem}
  
 Finally, recall that, for $k<l<k+2p(p-1)-1$, the only non-trivial $p$-components of the
 stable homotopy groups $\hos(S^l,S^k)$ are 
 $\hos_p(S^{k+q_s},S^k)=\mZ/p$, where $1\le s<p$ and $q_s=2s(p-1)-1$ \cite{tod}.

 \section{Moore polyhedra}
 \label{s2}
 
 The only atoms in $\cS_p^2$ are \emph{Moore atoms} $M_k$ ($k\in\mN$) which are cones of the maps 
 $S^2\xarr{p^k}S^2$. We denote their $d$-dimensional suspensions $M_k[d-3]$ by $M^d_k$ and call them
 \emph{Moore polyhedra}. For unification, we denote $S^d$ by $M^d_0$. We need to know the morphism groups 
 $\kM^{dr}_{kl}=\hos_p(M^r_l,M^d_k)$. We always suppose that $d-1\le r<d+2p-1$. Obviously, 
 $\kM^{dd}_{00}=\mZ_p$, $\kM^{d,d+2p-3}_{00}=\mZ/p$ and $\kM^{dr}_{00}=0$ if $r\notin\set{d,d+2p-3}$.
 If $k>0$, from the cofibre sequences
\begin{equation}\label{Ekd} 
   S^{d-1}\xarr{p^k} S^{d-1}\to M^d_k\to S^d\xarr{p^k} S^d  \tag{$\bE_k^d$}
 \end{equation}
 one easily obtains that $\kM^{dr}_{0k}=\kM^{dr}_{k0}=0$, except the cases
 \begin{align*}
 \kM^{d,d-1}_{k0}&\simeq\kM^{dd}_{0k}\simeq\mZ/p^k,\\
 \kM^{d,d+2p-3}_{k0}&\simeq\kM^{d,d+2p-3}_{0k}\simeq\\&\simeq\kM^{d,d+2p-4}_{k0}
 	\simeq\kM^{d,d+2p-2}_{0k}\simeq\mZ/p
 \end{align*}
 The values of $\kM^{dr}_{kl}$ for $k,l\in\mN,\,d-1\le r<d+2p-1$ can be obtained 
 if we apply     $\hos_p(\kM^r_l,\_\,)$ to the cofibre sequences \eqref{Ekd}. It gives exact 
 sequences 
\[
  \kM^{d-1,r}_{0l}\xarr{p^k} \kM^{d-1,r}_{0l}\to \kM^{dr}_{kl}\to \kM^{dr}_{0l}\xarr{p^k} \kM^{dr}_{0l},
\]  
 whence we get
 \begin{equation}\label{e21} 
 \kM^{dr}_{kl}=\begin{cases}
  \mZ/p^{\min(k,l)} & \text{if } r\in\set{d-1,d},\\
  \mZ/p & \text{if } r\in\set{d+2p-2,d+2p-4},\\
  \mZ/p\+\mZ/p &\text{if } r=d+2p-3,\\
  0 &\text{in other cases},
 \end{cases}
 \end{equation}
 The only non-trivial value here is for $r=d+2p-3$: we need to know that the exact sequence 
 \begin{equation}\label{e22} 
  0\to\mZ/p\xarr\al\kM^{d,d+2p-3}_{kl}\xarr\be\mZ/p\to0
  \end{equation} 
 splits. It splits indeed for $k=1$ since the middle term 
 is a module over $\kM^{dd}_{11}=\mZ/p$. If $k>1$, suppose that the sequence for $\kM^{d,d+2p-3}_{k-1,l}$ 
 splits. The commutative diagram
 \begin{equation}\label{Mk} 
 \begin{CD}
 S^{d-1}@>{p^k}>> S^{d-1}@>>> M^d_k@>>> S^d @>{p^k}>> S^d  \\
   @VpVV   @V1VV  @VVV @VpVV  @V1VV \\
 S^{d-1}@>{p^{k-1}}>> S^{d-1}@>>> M^d_{k-1}@>>> S^d @>{p^{k-1}}>> S^d   
 \end{CD}
 \end{equation}
 induces the commutative diagram
 \[
\begin{CD}
  0@>>> \mZ/p @>>> \kM^{d,d+2p-3}_{kl} @>>> \mZ/p @>>> 0 \\
  && @V1VV	@VVV  @V0VV \\
   0 @>>> \mZ/p @>>>  \kM^{d,d+2p-3}_{k-1,l} @>>> \mZ/p @>>> 0
 \end{CD} 
 \]
 Since the second row splits, the first one splits as well. Therefore, the sequence \eqref{e22} splits
 for all values of $k$ and $l$. 
 
 \begin{defin}\label{al}
   We fix generators of the groups $\kM^{dr}_{kl}$ and denote, for $r=d+2p-3$,
 
 by $\al^{d^*_*}_{kl}$ ($k,l\in\mN$) the generator of $\kM^{d+1,r+1}_{kl}$ which is in the 
 image of the map $\al$ from \eqref{e22};
 
 by $\al^{d}_{kl}$ ($k,l\in\mno$) the generator of $\kM^{dr}_{kl}$ which is not in $\im\al$;

 by $\al^{d^*}_{kl}$ ($k\in\mno,l\in\mN$) the generator of $\kM^{d,r+1}_{kl}$;
 
 by $\al^{d_*}_{kl}$ ($k\in\mN,l\in\mno$) the generator of $\kM^{d+1,r}_{kl}$;
 
 by $\ga^d_{kl}$ ($k,l\in\mno$) the generator of $\kM^{dd}_{kl}$;
 
 by $\ga^{d*}_{kl}$  ($k\in\mN,l\in\mno$) the generator of $\kM^{d+1,d}_{kl}$,
 \end{defin} 
 \noindent
 Note that all these morphisms are actually induced by maps $S^r\to S^d$.
 Using diagrams of the sort \eqref{Mk},
 one easily verifies that these generators can be so chosen that 
 \begin{equation}\label{e23} 
 \begin{split}
 & \al^{d^*_*}_{kl}\ga^{r+1}_{ll'}=\begin{cases}
 \al^{d^*_*}_{kl'}& \text{ if }\ l\le l',\\
 0 & \text{ if } l>l',
 \end{cases}\\    
& \al^{d^*}_{kl}\ga^{r+1}_{ll'}=\begin{cases}
 \al^{d^*}_{kl'}& \text{ if }\ l\le l',\\
 0 & \text{ if } l>l',
 \end{cases}\\    
 & \al^{d}_{kl}\ga^r_{ll'}=\begin{cases}
 \al^d_{kl'}& \text{ if }\ l\ge l' \text{ or } l=0,\\
 0 & \text{ if } 0<l<l',
 \end{cases}\\
 & \al^{d_*}_{kl}\ga^r_{ll'}=\begin{cases}
 \al^{d_*}_{kl'}& \text{ if }\ l\ge l' \text{ or } l=0,\\
 0 & \text{ if } 0<l<l',
 \end{cases}\\
 & \al^{d^*}_{kl}\ga^{r*}_{lk'}=\al^{d}_{kk'},\\
 & \al^{d^*_*}_{kl}\ga^{r*}_{lk'}=\al^{d_*}_{kk'},\\
 & \ga^{d+1}_{k'k}\al^{d^*_*}_{kl}=\begin{cases}
 \al^{d^*_*}_{k'l}& \text{ if }\ k\ge k',\\
 0 & \text{ if } k<k',
 \end{cases}\\    
 & \ga^{d+1}_{k'k}\al^{d_*}_{kl}=\begin{cases}
 \al^{d_*}_{k'l}& \text{ if }\ k\ge k',\\
 0 & \text{ if } k<k',
 \end{cases}\\    
 & \ga^d_{k'k}\al^{d}_{kl}=\begin{cases}
 \al^d_{k'l}& \text{ if }\ k\le k',\\
 0 & \text{ if } k>k',
 \end{cases}\\  
 & \ga^d_{k'k}\al^{d^*}_{kl}=\begin{cases}
 \al^{d^*}_{k'l}& \text{ if }\ k\le k',\\
 0 & \text{ if } k>k',
 \end{cases}\\  
 & \ga^{d*}_{k'k}\al^{d}_{kl}=\al^{d_*}_{k'l},\\
 & \ga^{d*}_{k'k}\al^{d^*}_{kl}=\al^{d_*^*}_{k'l}.
 \end{split}
 \end{equation}
 (always $r=d+2p-3$).
 
 \section{Atoms in $\cS_p^{2p-1}$}
 \label{s3} 
 
 For $n\le{2p-1}$ the description of the category $\cS_p^n$ is very simple. First, the next
 fact is rather obvious.
 
 \begin{prop}
 If $n<2p-1$, all indecomposable polyhedra in $\cS^n_p$ are Moore spaces $M^d_k$. In particular,
 $M^2_k$ are atoms in $\cS^2_p$ and there are no atoms in $\cS^n_p$ if $2<n<2p-1$.
 \end{prop}
 \begin{proof}[Proof \rm{} is an easy induction.]
  For $n=2$ it is known. Suppose that $2<n<2p-1$ and the claim is true for 
  $\cS^{n-1}_p$. We use Theorem~\ref{12} with $m=2n-2$. Then $\cA$ consists of wedges of the sphere 
  $S^{2n-2}$, while the spheres $S^d\ (n\le d\le 2n-2)$ and the Moore atoms $M^d_k\ (n<d\le 2n-2)$ form a 
  set of additive generators of $\cB$. Note that in our case $\kM^{dr}_{k0}=0$ for $n<d\le r\le2n-2$, 
  except $\kM^{2n-2,2n-2}_{00}$. Therefore, the only new indecomposable polyhedra in $\cS^n_p$ are the 
  Moore spaces $M^{2n-1}_k$, which are not atoms.
 \end{proof}
 
 Consider the category $\cS^{2p-1}_p$. Again we use Theorem~\ref{12} with $m=2n-3=4p-5$. Now
 a set of additive generators of $\cA$ is  \\
 \centerline{$\bA=\set{S^{4p-4}=M^{4p-4}_0,\,S^{4p-5}=M^{4p-5}_0,\,M^{4p-5}_k}$,} \\
 and a set of additive generators of $\cB$ is \\
 \centerline{$\bB=\set{S^d=M^d_0\ (2p-1\le d\le 4p-5),\ M^d_k\ (2p-1<d\le 4p-5)}$.}
 The only non-zero values of $\hos_p(A,B)$, where $A\in\bA,\,B\in\bB$, are 
 
 $\kM^{2p,4(p-1)}_{kl}\simeq\mZ/p$, with generators 
 $\al^{(2p-1)_*}_{kl}$ $(k\in\mN,\,l\in\mno)$, 
 
 $\kM^{2p-1,4(p-1)}_{0l}\simeq\mZ/p$ with generators $\al^{2p-1}_{0l}$ $(l\in\mno)$,
 
 $\kM^{4p-5,4p-5}_{00}=\mZ_p$ with generator $\ga^{4p-5}_{00}$.\\
 Therefore, the matrix $F$ defining a morphism $f:A\to B$ ($A\in\cA,\,B\in\cB$) is a direct sum
 $F'\+F''$, where $F''$ is with coefficients from $\kM^{4p-5,4p-5}_{00}$ and $F'$ is a block
 matrix $(F_{kl})_{k,l\in\mno}$, where $F_{kl}$ is with coefficients from $\kM^{2p,4(p-1)}_{kl}$
 if $k\ne0$ and $F_{0l}$ is with coefficients from $\kM^{2p-1,4(p-1)}_{0l}$. We denote by $F_k$ the
 horizontal stripe $(F_{kl})_{l\in\mno}$ with fixed $k$ 
 and by $F^l$ the vertical stripe $(F_{kl})_{k\in\mno}$ with
 fixed $l$. Morphisms between objects from $\bA$ and $\bB$ act according to the rules \eqref{e23}. 
 They imply that two matrices $F$ and $G$ of such structure define isomorphic objects from $\cE(\cM)$ 
 \iff $G''=TF''T'$ for some invertible matrices $T,T'$ over $\mZ_p$ and $F'$ can be transformed to $G'$ 
 by a sequence of the following transformations:
 
 $F_k\mapsto TF_k$, where $T$ is an invertible matrix over $\mZ/p$;
 
 $F^l\mapsto F^lT'$, where $T'$ is an invertible matrix over $\mZ/p$;
 
 $F_k\mapsto F_k+UF_{k'}$, where $k'>k$ or $k'=0,\,k\ne0$ and $U$ is any matrix of
 appropriate size over $\mZ/p$;
 
 $F^l\mapsto F^l+F^{l'}U'$, where $l'<l$ and $U'$ is any matrix of
 appropriate size over $\mZ/p$.\\
 Using these transformations one can easily make the matrix $F''$ diagonal and reduce $F'$ to a matrix
 having at most one non-zero element in each row and in each column. Then the corresponding object from
 $\cE(\cM)$ splits into direct sum of objects given by $1\xx1$ matrices. The $1\xx1$ matrices
 over $\kM^{4p-5,4p-5}_{00}$ give Moore polyhedra $M^{4p-4}_t$, which are not atoms 
 (and belong to $\cA$).
 Therefore, the atoms in $\cS^{2p-1}_p$ are $C_{kl}$ $(k,l\in\mno)$ corresponding to the $1\xx1$
 matrices $(\al^{(2p-1)_*}_{kl})$ if $k\ne0$ and to $(\al^{2p-1}_{0l})$ if $k=0$. We call these 
 polyhedra \emph{Chang atoms}, in analogy with \cite{ba}. They are defined by the cofibration
 sequences 
  \begin{equation}\label{Ckl} 
  \begin{split}
   M^{4p-4}_l\to M^{2p}_k\to C_{kl}\to M_l^{4p-3} \to M^{2p+1}_k \quad\text{if } k\ne0,\\  
   M^{4p-4}_l\to S^{2p-1}\to C_{0l}\to M_l^{4p-3} \to S^{2p} \quad\text{if } k=0.
  \end{split}
  \tag{$\bC_{kl}$}
 \end{equation}
  We can also present Chang atoms by their \emph{gluing diagrams}, as in \cite{ba,d1,d2}:
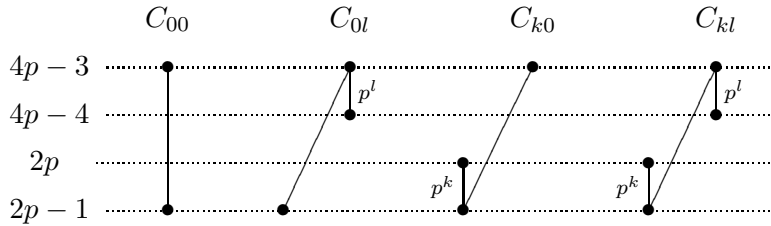
 \[
 \xymatrix@R=.5ex@C=1em{
 	& {C_{00}}  &&& {C_{0l}} &&& {C_{k0}} &&& {C_{kl}} \\
  	 {4p-3\ } \ar@{.} [rrrrrrrrrrr] & \cel \ar@{-} [ddd] &&& \cel \ar@{-}[dddl] &&& 
	\cel \ar@{-}[dddl]  &&& \cel \ar@{-}[dddl] \ar@{-}[d]^{p^l} & {} \\
       {4p-4\ } \ar@{.} [rrrrrrrrrrr] &&&& \cel \ar@{-}[u]_{p^l} && 
	 &&& &\cel&{} \\
 {\ 2p\ \ \ }  \ar@{.} [rrrrrrrrrrr] &&&&&& \cel \ar@{-}[d]_{p^k} &&& \cel\ar@{-}[d]_{p^k} & & {}\\
        {2p-1\ } \ar@{.} [rrrrrrrrrrr] & \cel && \cel &&& \cel &&& \cel&&  {} 	
  } 
 \]
 Here bullets correspond to cells, lines show the attaching maps and these maps are
 specified if necessary. 

 Theorem~\ref{12} and cofibration sequences \eqref{Ckl} easily give the following values of the 
 endomorphism rings of Chang atoms modulo the ideal $\cJ$:
\begin{align*}
 \De&=\setsuch{(a,b)}{a\equiv b\!\pmod p}\sb\mZ_p\xx\mZ_p\ \text{ for }\ C_{00};\\
 \De_k&=\setsuch{(a,b)}{a\equiv b\!\pmod p}\sb\mZ_p\xx\mZ/p^k\ \text{ for }\ C_{0k} 
   \text{ and } C_{k0}\ (k\ne0);\\
\De_{kl}&=\setsuch{(a,b)}{a\equiv b\!\pmod p}\sb\mZ/p^k\xx\mZ/p^l\ \text{ for }\ C_{kl}\ (k\ne0,l\ne0).
\end{align*}
 Since all these rings are local and $\cJ^2=0$, the endomorphism rings of Chang atoms are local.
 Therefore, these polyhedra are indeed indecomposable (hence atoms). Moreover, we can use the
 unique decomposition theorem of Krull--Schmidt--Azumaya \cite[Theorem I.3.6]{bass} and obtain 
 the final result.
 
 \begin{theorem}\label{31} 
 The atoms in $\cS^{2p-1}_p$ are Chang atoms $C_{kl}$ \bl$k,l\in\mno$\br. 
 Every polyhedron from $\cS^{2p-1}_p$ uniquely decomposes into a wedge of spheres, Moore polyhedra
 and Chang atoms.
 \end{theorem}
 
 In Section \ref{s5} we will need the whole endomorphism ring of the atom $C=C_{00}$. 
 Applying $\hos_p$ to the sequence ($\bC_{00}$) as below, we obtain the commutative diagram 
 with exact columns and rows
 \begin{equation}\label{e31} 
 \xymatrix@R=1em@C=1.5em{
   &  S^{2p} \arb[r] & S^{4p-3}\arb[r] & C \arb[r] & S^{2p-1} \arb[r] & S^{4p-4} \\
 S^{4p-4} \ar[d] & \oo\ar[r]\ar[d] & \oo \ar[r]\ar[d] & \oo \ar[r]\ar[d] & \oo
     \ar[r]\ar[d] & \mZ_p \ar[d]_s  \\
 S^{2p-1}\ar[d] & \oo \ar[r]\ar[d] & \oo \ar[d]\ar[r] & {p\mZ_p}\ar[r]\ar[d] & \mZ_p
    \ar[r]^s\ar[d]_1 & \mZ/p \ar[d] \\
 C \ar[d] & \oo \ar[r]\ar[d] & {p\mZ_p}\ar[r]\ar[d] & \hos_p(C,C) \ar[r]\ar[d] &
   {\mZ_p}\ar[r]\ar[d] & \oo \ar[d] \\
 S^{4p-3}\ar[d] & \oo \ar[r]\ar[d] & {\mZ_p}\ar[r]^1\ar[d]_s & {\mZ_p}\ar[r]\ar[d] &
   \oo \ar[r]\ar[d] & \oo\ar[d] \\
 S^{2p} & \mZ_p \ar[r]^s & \mZ/p \ar[r] & \oo \ar[r] & \oo \ar[r] & \oo
  }
 \end{equation}
 where $s$ marks surjections. The central row and the central column, corresponding to the 
 polyhedron $C$, are easily calculated from all other values. It shows that $\hos_p(C,C)$ has
 no torsion, hence coincides with $\De$. Analogous calculations show that $\cJ(C_{kl},C_{kl})$ 
 equals $\mZ/p$ if $k=0$ or $l=0$ (but not both) and $(\mZ/p)^2$ if both $k\ne0$ and $l\ne0$.
 
 Theorem~\ref{31} also gives a description of \emph{genera} of $p$-primary polyhedra in $\cS^{2p-1}$.
 Recall that a \emph{genus} is a class of polyhedra such that all their localizations are isomorphic
 (in the corresponding localized categories). Certainly, if these polyhedra are $p$-primary, we 
 only need to compare their $p$-localizations. 
 Equivalently, two polyhedra $X,Y$ are in the same genus \iff there is a
 wedge of spheres $W$ such that $X\vee W\simeq Y\vee W$ in $\cS$ \cite[Theorem 2.5]{dk}. 
 Let $g(X)$ be the number of isomorphism
 classes of polyhedra in the genus of $X$. If $\La=\hos(X,X)/\tors(X)$, where $\tors(X)$ is the torsion 
 part of $\hos(X,X)$, then $\mQ\*\La$ is a semi-simple $\mQ$-algebra, so there is a maximal order
 $\Ga\spe\La$ in this algebra. Then $\La\spe m\Ga$ for some positive integer $m$ and $g(X)=g(\La)$ 
 equals the number of cosets
 \[
  \im\ga\backslash (\Ga/m\Ga)^\xx/(\La/m\La)^\xx,
 \]
 where $R^\xx$ denotes the group of invertible elements of a ring $R$ and $\ga$ is the natural
 map $\Ga^\xx\to(\Ga/m\Ga)^\xx$ \cite[Section 3]{dk}. If $X=C_{0k}$ or $X=C_{k0}$, then $\La=\mZ$; if 
 $X=C_{kt}$, then $\La=0$. So $g(X)=1$ for all these cases. For $X=C$ this formula
 implies that $g(C)=(p-1)/2$. If $\nu\in\hos_p(S^{4p-4},S^{2p-1})$ is an element of order $p$,
 the polyhedra from the genus of $C$ can be realized as the cones $C(c)$ of the maps 
 $S^{4p-4}\xarr{c\nu}S^{2p-1}$ for $1\le c\le(p-1)/2$.
 
  \section{Atoms in $\cS_p^n$ for $2p\le n\le 4(p-1)$}
 \label{s4} 
  Let now $2p\le n\le 4(p-1)$. We use Theorem~\ref{12} with $m=n+2p-3$. Then $\cA$ has a set of 
  additive generators\\
 \centerline{$\bA=\set{S^r\ (m\le r<2n-1),\ M^r_l\ (m<r<2n-1,\,l\in\mN}$,}\\
 and $\cB$ has a set of additive generators \\
 \centerline{$\bB=\set{S^d\ (n\le d\le m),\ M^d_k\ (n<d\le m,\,k\in\mN)}$.}
 Morphisms $\vi:A\to B$, where $A\in\cA,\,B\in\cB$, are given by block matrices such that their blocks
 have coefficients from $\kM^{dr}_{kl}$. Taking into consideration Definition~\ref{al}, it is 
 convenient to denote these blocks as follows.
  \begin{defin}\label{41} 
 We introduce sets 
 \begin{align*}
 \dE^\circ&=\left\{e^d_k\ (n<d\le 2(n-p)+1,k\in\mno),\right.\\ 
 &\left.\qquad e^{d*}_k\ (n\le d\le 2(n-p),k\in\mN),\ e^n_0,e^m_0\right\},
 \\
 \dF^\circ&=\left\{f^d_l\ (n<d\le 2(n-p)+1,l\in\mno),\right.\\
 &\left.\qquad f^{d*}_l\ (n\le d\le 2(n-p),l\in\mN),\ f^n_0\right\}, 
 \end{align*}
 and consider a morphism $\vi:A\to B$, where $A\in\cA,\,B\in\cB$, as a block matrix 
 $(\Phi_{ef})_{e\in\dE^\circ, f\in\dF^\circ}$. Namely, 
\\ 
 - \ the block $\Phi_{e^d_k,f^d_l}$ consists of coefficients at $\al^d_{kl}$;
 \\ 
 - \  the block $\Phi_{e^{d*}_k,f^d_l}$ consists of coefficients at $\al^{d_*}_{kl}$;
  \\
- \  the block $\Phi_{e^d_k,f^{d*}_l}$ consists of coefficients at $\al^{d^*}_{kl}$;
  \\
 - \ the block $\Phi_{e^{d*}_k,f^{d*}_l}$ consists of coefficients at $\al^{d^*_*}_{kl}$;
  \\
 - \ the block $\Phi_{e^m_0,f^n_0}$ consists of coefficients at $\ga^m_{00}$.\\
 Note that for $n=4(p-1)$ we need not specially add $e^m_0$ to $\dE^\circ$, since $m=2(n-p)+1$ in this case.
 
 We also denote by $\Phi_e$ for a fixed $e\in\dE^\circ$ the horizontal stripe 
 $(\Phi_{ef})_{f\in\dF^\circ}$ and 
 by $\Phi^f$ for a fixed $f\in\dF^\circ$ the vertical stripe $(\Phi_{ef})_{e\in\dE^\circ}$.
 \end{defin}
 \noindent
 Note that the horizontal stripes $\Phi_{e^d_k}$ and $\Phi_{e^{(d+1)*}_k}$ have the same number of rows
 and the vertical stripes $\Phi^{f^d_l}$ and $\Phi^{f^{(d+1)*}_l}$ have the same number of columns.
 All blocks $\Phi_{ef}$ defined above have coefficients from $\mZ/p$, except $\Phi_{e^m_0,f^n_0}$ which
 has coefficients from $\mZ_p$.
 
 Using automorphisms of $S^m$ we can make the block $\Phi_{e^m_0,f^n_0}$ diagonal with powers of $p$ or
 zero on diagonal. So we always suppose that it is of this shape and exclude this block from the matrix
 $\Phi$. Then we have to split the remaining part of the vertical stripe $\Phi^{f^n_0}$ and, 
 if $n=4(p-1)$, of the horizontal stripe $\Phi_{e^m_0}$ into several 
 stripes, respectively, $\Phi^{f^{n,s}_0}$ and $\Phi_{e^{m,s}_0}$, where the indices $s\in\mni$ 
 correspond to diagonal entries $p^s$ (setting $p^\8=0$). Respectively, we modify the sets 
 $\dE^\circ$ and $\dF^\circ$. Namely, we denote 
 \begin{equation}\label{e41} 
 \begin{split}
 \dF&=(\dF^\circ\setminus\{f^n_0\})\cup\setsuch{f^{n,s}_0}{s\in\mni},\\
 \dE&=\dE^\circ\setminus\{e^m_0\}\ \text{ if } n<4(p-1),\\
 \dE&=(\dE^\circ\setminus\{e^m_0\})\cup\setsuch{e^{m,s}_0}{s\in\mni}\ \text{ if } n=4(p-1).
 \end{split}
 \end{equation}
 Note that, if $n=4(p-1)$, the number of rows in the horizontal stripe $\Phi_{e^{d,s}_0}$ 
 with $s\ne\8$ equals the number of columns in the vertical stripe $\Phi^{f^{d,s}_0}$. We split
 the sets $\dE$ and $\dF$ according to the upper indices. Namely, $\dE_d$ consists
 of all elements from $\dE$ with the upper index $d$, $d^*$ or, if $d=m$, $(m,s)$; $\dF_d$ 
 consists of all elements from $\dF$ with the upper index $d$, $d^*$ or, if $d=n$, $(n,s)$.
 We define a linear order on each $\dE_d$ and $\dF_d$ setting

 $e^d_k<e^d_{k'}$ and $e^{d*}_k>e^{d*}_{k'}$ if $k<k'$, and $e^d_k<e^{d*}_{k'}$ for all $k,k'$;
 
 if $n=4(p-1)$, then $e^{m,s}_0<e^{m,s'}_0<e^m_k$ for $s>s'$ and any $k\in\mN$;
 
 $f^d_k<f^d_{k'}$ and $f^{d*}_k>f^{d*}_{k'}$ if $k<k'$ or $k>k'=0$, and $f^d_k<f^{d*}_{k'}$ for all 
 $k,k'$;
 
 $f^m_k<f^{m,s}_0<f^{m,s'}_0$ for $s<s'$ and any $k\in\mN$.
 
 The formulae \eqref{e23} imply that two such block matrices $\Phi$ and $\Phi'$ define isomorphic
 objects from $\cE(\cM)$ \iff $\Phi$ can be transformed to $\Phi'$ by a sequence of the following
 transformations:
 
 $\Phi_e\mapsto T_e\Phi_e$, where $T_e$ are invertible matrices and $T_{e^{d*}_k}=T_{e^{d+1}_k}$ for all
 possible values of $d,k$;
 
 $\Phi^f\mapsto \Phi^fT^f$, where $T^f$ are invertible matrices and $T^{f^{d*}_k}=T^{f^{d+1}_k}$ for all
 possible values of $d,k$; 
 
 if $n=4(p-1)$, then, moreover, $T_{e^{m,s}_0}=T^{f^{n,s}_0}$ for all $s\in\mN$ (not for $s=\8$);
 
 $\Phi_e\mapsto U_{ee'}\Phi_{e'}$ if $e'<e$, where $U_{ee'}$ is an arbitrary matrix of the 
 appropriate size;

 $\Phi^f\mapsto \Phi^{f'}U^{f'f}$ if $f'>f$, where $U^{f'f}$ is an arbitrary matrix of the 
 appropriate size.
 
 These rules show that the classification of polyhedra in $\cS_p^n$ actually coincides with the
 classification of representations of the \emph{bunch of chains} 
 $\dX=\set{\dE_d,\dF_d,<,\sim\, \mid n\le d\le m}$ 
 (cf. \cite{bo} or \cite[Appendix B]{bu}), where the relation $\sim$ is defined by the exclusive rules:
 \begin{align*}
 & e^{d*}_k\sim e^{d+1}_k\, \text{ and }\, f^{d*}_k\sim f^{d+1}_k\, 
 \text{ for }\, n<d\le2(n-p),\,k\in\mN,
 \intertext{ and, if $n=4(p-1)$,}
 & e^{m,s}_0\sim f^{n,s}_0\, \text{ for }\, s\in\mN\ \,(\text{not for }\,s=\8).
 \end{align*}
 Thus the description of indecomposable representations given in \cite{bo,bu} implies a description
 of indecomposable polyhedra from $\cS^n_p$. Recall the necessary
 combinatorics. We write $e-f$ and $f-e$ if $e\in\dE_d$ and $f\in\dF_d$ (with the same $d$) and set
 $|\dX|=\dE\cup\dF$.
 
 \begin{defin}\label{word} 
 \begin{enumerate}
 \item A \emph{word} is a sequence $w=x_1r_1x_2r_2\dots x_{l-1}r_{l-1}x_l$, where $x_i\in|\dX|$, 
 $r_i\in\set{-,\sim}$ such that
  \begin{enumerate}
  \item $r_i\ne r_{i+1}$ for all $1\le i<l-1$;
  \item $x_ir_ix_{i+1}$ ($1\le i<l$) according to the definition of the relations $\sim$ and 
  $-$ given above;
  \item if $r_1=-$ ($r_{l-1}=-$), then $x_1\nsim y$ for all $y\in|\dX|$ (respectively, $x_l\nsim y$
  for all $y\in|\dX|$). 
  \end{enumerate} 
  We say that $l$ is the \emph{length of the word} $w$ and write $l=\ln w$.
 \item For a word $w$ as above we denote by $\rE(w)=\setsuch{i}{1\le i\le l,\,x_i\in\dE}$ and
 $\rF(w)=\setsuch{i}{1\le i\le l,\,x_i\in\dF}$.
 \item The \emph{inverse word} $w^*$ of the word $w$ is the word $x_lr_{l-1}x_{l-1}\dots r_2x_2r_1x_1$.
 \item A word $w$ is said to be a \emph{cycle} if $r_1=r_{l-1}=\sim$ and $x_l-x_1$. Then we set $r_l=-$,
 $x_{i+ql}=x_i$ and $r_{i+ql}=r_i$ for all $q\in\mZ$ (in particular, $r_0=-$). 
 \item The \emph{$k$-th shift} of a cycle $w$, where $k$ is an even integer, is the cycle 
 $w^{[k]}=x_{k+1}r_{k+1}\dots r_{k-1}x_k$ (obviously, it is enough to consider $0\le k<l$). 
 \item A cycle $w$ is said to be \emph{non-periodic} if $w\ne w^{[k]}$ for $0<k<l$.
 \item For a cycle $w$ and an integer $0<k<l$ we denote by $\nu(k,w)$ the number of even integers
 $0<i<k$ such that both $x_i$ and $x_{i-1}$ belong either to $\dE$ or to $\dF$.
 \end{enumerate}
 Note that, since $x\nsim x$ for all $x\in|\dX|$, there are no symmetric words and symmetric cycles
 in the sense of \cite[Appendix B]{bu}.
 \end{defin} 
 
 To words and cycles correspond indecomposable representations of the bunch of chains $\dX$ called
 \emph{strings} and \emph{bands}. We describe the corresponding matrices $\Phi$ (recall that we have
 already excluded the part $\Phi_{e_mf_n}$). 
 
 \begin{defin}\label{string} 
 \begin{enumerate}
 \item  If $w$ is a word, the corresponding 
 \emph{string matrix} $\Phi(w)$ is constructed as follows:
\\ 
 - \ its rows are labelled by the set $\rE(w)$ and its columns are 
 labelled by the set $\rF(w)$;
 \\
 - \ the only non-zero entries are those at the places $(i,i+1)$ if $r_i=-$ and $i\in\rE(w)$ and
 $(i+1,i)$ if $r_i=-$ and $x_i\in\rF(w)$; they equal $1$.
 \\
 We denote the corresponding polyhedron by $A(w)$ and call it a \emph{string polyhedron} whenever it
 does not coincide with a sphere, a Moore or a Chang polyhedron.\!%
 \footnote{\,The words consisting of one letter $x$ correspond to spheres, the words of the form 
 $x\sim y$ correspond to Moore polyhedra, the words that only have one symbol `$-$' correspond to 
 Chang polyhedra, and these are all exceptions.}

 \item If $w$ is a non-periodic cycle, $z\in\mN$ and $\pi\ne t$ is a unital 
 irreducible polynomial of
 degree $v$ from $(\mZ/p)[t]$, the \emph{band matrix} $\Phi(w,z,\pi)$ is a block matrix, 
 where all blocks are of size $zv\xx zv$, constructed as follows:
 \\
 - \ its horizontal stripes are labelled by the set $\rE(w)$ and its
 vertical stripes are labelled by the set $\rF(w)$;
 \\
 - \ the only non-zero blocks are those at the places $(i,i+1)$ if $r_i=-$ and $i\in\rE(w)$ and
 $(i+1,i)$ if $r_i=-$ and $i\in\rF(w)$ (note that here $i=l$ is also possible);
 \\
 - \ these non-zero blocks equal $I_{zv}$ (the identity $zv\xx zv$ matrix), except the block at the place
 $(l1)$ (if $l\in\rE(w)$) or $(1l)$ (if $l\in\rF(w)$) which is the Frobenius matrix with the
 characteristic polynomial $\pi^v$. If $\pi=t-c$ is linear, we replace the Frobenius matrix by
 the Jordan $z\xx z$ block with the eigenvalue $c$.
 \\
 We denote the corresponding polyhedron by $A(w,z,\pi)$ and call it a \emph{band polyhedron}.\!%
 \footnote{\,Band polyhedra never coincide with spheres, Moore or Chang polyhedra.}
 \end{enumerate}
 \end{defin}
 
 Using these notions, we obtain the following description of polyhedra in the category $\cS^n_p$.
 
 \begin{theorem}\label{main} 
 \begin{enumerate}
 \item  All string and band polyhedra are indecomposable and every indecomposable polyhedron from
 $\cS^n_p$, except spheres, Moore and Chang polyhedra, is isomorphic to a string or band polyhedron.
 
 \item  The only isomorphisms between string and band polyhedra are the following:
   \begin{enumerate}
   \item  $A(w)\simeq A(w^*)$;
   \item  $A(w,z,\pi)\simeq A(w^*,z,\pi)$;
   \item  $A(w,z,\pi)\simeq A(w^{[k]},z,\pi^*)$, where $\pi^*=\pi$ if $\nu(k,w)$ is even and
   $\pi^*(t)=t^z\pi(0)^{-1}\pi(1/t)$ if $\nu(k,w)$ is odd.\!%
 \footnote{\,If $\pi=t^v+a_1t^{v-1}+\dots+a_{v-1}t+a_v$, then 
 $\pi^*=t^v+a_v^{-1}(a_{v-1}t^{v-1}+\dots+a_1t+1)$.}
   \end{enumerate}
   
 \item  Endomorphism rings of string and band polyhedra are local, hence every polyhedron from
 $\cS^n_p$ uniquely decomposes into a wedge of spheres, Moore and Chang polyhedra, and 
 string and band polyhedra.
 
 \item  A string or band polyhedron is an \emph{atom} in $\cS^n_p$ \iff the corresponding word contains
 at least one letter from $\dE_d$ and at least one letter from $\dF_{2(n-p)+1}$. 
 \end{enumerate}
 \end{theorem}
 
 Note that in this case we can simplify the writing of the words, since for every $x\in|\dX|$ there is
 at most one element $y\in|\dX|$ such that $x\sim y$ and then $x-y$ is impossible. Hence we can omit all
 symbols $-$ and write $x$ instead of $x\sim y$. For instance, 
 $e^d_kf^{d-1}_le^{(d-2)*}_{k'}f^{d-1}_{l'}$ means 
 $e^d_k\sim e^{(d-1)*}_k-f^{d-1}_l\sim f^{(d-2)*}_l-e^{d-2}_{k'}
 \sim e^{(d-1)*}_{k'}-f^{d-1}_{l'}\sim f^{(d-2)*}_{l'}$.
 One can prove that there can be at most one place in a word $w$ where a fragment $e^{m,s}\sim f^{n,s}$ or 
 $f^{n,s}_0\sim e^{m,s}_0$ occurs; moreover, if it occurs, $w$ cannot be a cycle.

 \begin{exam}\label{45} 
 We give several examples of string and band polyhedra and their gluing diagrams. In these examples we
 suppose that $p=3$.
 \begin{enumerate}
 \item The ``smallest'' possible string atoms are for $n=6$. They have 3 cells and are given by the words
 $e^{6*}_kf^7_0$ or $e^6_0f^{6*}_l$. The smallest band atoms have 4 cells. They are $A(w_0,1,t\mp1)$,
 where $w_0=e^7_kf^7_l$. Here are their gluing diagrams:
 
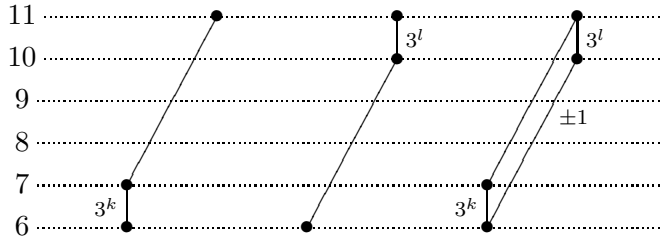
\[
  \xymatrix@R=.5ex@C=2.5em{
 11 \ar@{.}[rrrrrrr] && \cel && \cel \ar@{-}[d]^{3^l} & & \cel \ar@{-}[d]^{3^l} & \\
 10 \ar@{.}[rrrrrrr] &&&& \cel && \cel &\\
 9 \ar@{.}[rrrrrrr] &&&&&&& \\
 8 \ar@{.}[rrrrrrr] &&&&&&& \\
 7 \ar@{.}[rrrrrrr] & \cel \ar@{-}[d]_{3^k} \ar@{-}[uuuur] &&&& 
 		\cel \ar@{-}[d]_{3^k} \ar@{-}[uuuur] && \\
 6 \ar@{.}[rrrrrrr] & \cel && \cel  \ar@{-}[uuuur] && \cel  \ar@{-}[uuuur]_(.7){\pm1} &&
 }  
 \]
 \item  More complicated band atoms are $A(w_0,1,t^2+1)$ and $A(w_0,2,t\mp1)$. Their gluing diagrams
 are
 
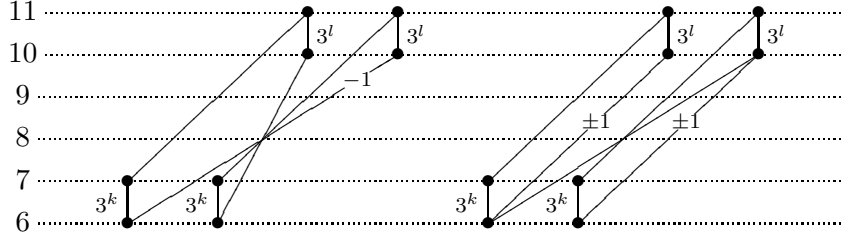
\[
  \xymatrix@R=.5ex@C=2.5em{
 11 \ar@{.}[rrrrrrrrr] &&& \cel\ar@{-}[d]^{3^l} & \cel\ar@{-}[d]^{3^l}
 && & \cel\ar@{-}[d]^{3^l} & \cel\ar@{-}[d]^{3^l} &\\
 10 \ar@{.}[rrrrrrrrr] &&& \cel &\cel &&& \cel &\cel &\\
 9 \ar@{.}[rrrrrrrrr] &&&&&&&&& \\
 8 \ar@{.}[rrrrrrrrr] &&&&&&&&& \\
 7 \ar@{.}[rrrrrrrrr] & \cel \ar@{-}[d]_{3^k} \ar@{-}[uuuurr]& \cel \ar@{-}[d]_{3^k} \ar@{-}[uuuurr] &&&
  \cel \ar@{-}[d]_{3^k} \ar@{-}[uuuurr] & \cel \ar@{-}[d]_{3^k} \ar@{-}[uuuurr] 
 		 &&& \\
 6 \ar@{.}[rrrrrrrrr] & \cel \ar@{-}[uuuurrr]|(.85){-1} & \cel \ar@{-}[uuuur] &&& 
 \cel  \ar@{-}[uuuurr]|(.6){\pm1}
 \ar@{-}[uuuurrr]& \cel \ar@{-}[uuuurr]|(.6){\pm1} &&&
 }  
 \]
 The non-trivial attachments of cells of dimension $10$ come, respectively, from the Frobenius
 matrix {\small$\mtr{0&-1\\1&0}$} and the Jordan block {\small$\mtr{\pm1&1\\0&\pm1}$}.
 \item For the maximal value $n=8$ the smallest atoms contain 4 cells. They are given by the words  $e^8_0f^{8,s}_0f^{11}_0$ and have the gluing diagrams
 
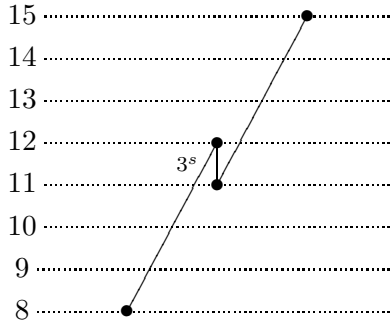
\[
 \xymatrix@R=.5ex@C=2.5em{
 15 \ar@{.}[rrrr] & && \cel &\\
 14 \ar@{.}[rrrr] &&&& \\
 13 \ar@{.}[rrrr] &&&& \\
 12 \ar@{.}[rrrr] && \cel \ar@{-}[d]_{3^s\ } && \\
 11 \ar@{.}[rrrr] && \cel \ar@{-}[uuuur] && \\
 10 \ar@{.}[rrrr] &&&& \\
  9 \ar@{.}[rrrr] &&&& \\
  8 \ar@{.}[rrrr] & \cel \ar@{-}[uuuur] &&&
  }
 \]
  \item The band atoms for $n=8$ are rather complicated and cannot be ``small''. For instance, one
  of the smallest is $A(w,1,t\mp1)$, where 
  $w=e^{8*}_{k_1}f^{9*}_{l_1}e^{10*}_{k_2}f^{11}_{l_2}e^{10}_{k_3}f^9_{l_3}$. The gluing diagram
  for this atom is
  
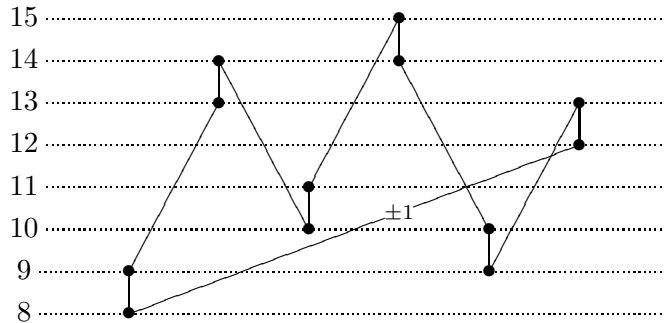
\[
 \xymatrix@R=.5ex@C=2.5em{
 15 \ar@{.}[rrrrrrr] && && \cel \ar@{-}[d]&&&\\
 14 \ar@{.}[rrrrrrr] && \cel \ar@{-}[d]\ar@{-}[ddddr] && \cel \ar@{-}[ddddr] &&&\\
 13 \ar@{.}[rrrrrrr] && \cel &&&& \cel \ar@{-}[d] &\\
 12 \ar@{.}[rrrrrrr] &&&&&& \cel & \\
 11 \ar@{.}[rrrrrrr] &&& \cel \ar@{-}[d] \ar@{-}[uuuur] &&&& \\
 10 \ar@{.}[rrrrrrr] &&& \cel && \cel \ar@{-}[d] && \\
  9 \ar@{.}[rrrrrrr] & \cel \ar@{-}[d]\ar@{-}[uuuur] &&&& \cel \ar@{-}[uuuur] &&\\
  8 \ar@{.}[rrrrrrr] & \cel \ar@{-}[uuuurrrrr]|(.6){\pm1} &&&&&&
  }
 \]
 (the powers of $3$ near vertical lines are omitted).
 \item Finally, we give an example of an atom having exactly one cell of each dimension (we do not 
 precise the corresponding word, since it can be easily restored).
 
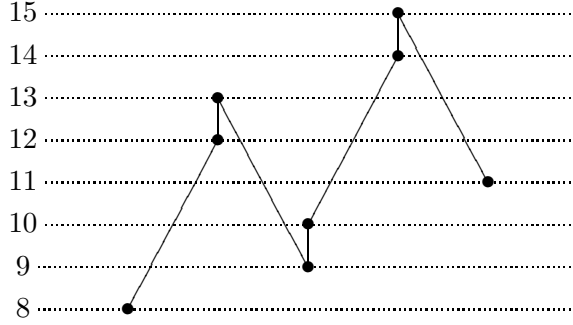
\[
 \xymatrix@R=.5ex@C=2.5em{
 15 \ar@{.}[rrrrrr] && && \cel \ar@{-}[d] \ar@{-}[ddddr] &&\\
 14 \ar@{.}[rrrrrr] && &&  \cel && \\
 13 \ar@{.}[rrrrrr] && \cel \ar@{-}[d] \ar@{-}[ddddr] &&&&\\
 12 \ar@{.}[rrrrrr] && \cel &&&& \\
 11 \ar@{.}[rrrrrr] &&&&& \cel & \\
 10 \ar@{.}[rrrrrr] && & \cel \ar@{-}[d] \ar@{-}[uuuur] &&& \\
  9 \ar@{.}[rrrrrr] && & \cel &&& \\
  8 \ar@{.}[rrrrrr] & \cel \ar@{-}[uuuur] &&&&&
	}
 \]
 Another atom with this property is the properly shifted $S$-dual of this one in the 
 sense of 
 \cite[Chapter 14]{sw}.
 \end{enumerate} 
 \end{exam}
 
 One can also calculate genera of $p$-primary polyhedra for $2p\le n\le 4(p-1)$. Namely, let $\La(X)$
 denotes the ring $\hos(X,X)/\tors(X)$. We call the end $x_1$ or $x_l$ of a word $w$ \emph{spherical}
 if it of the form $e^d_0$ or $f^d_0$. Note that these letters can only occur at an end of the word
 since they are not related by $\sim$ to any letter. It is rather easy to verify that $\La(X)=0$ if
 $X$ is a band polyhedron, while for a string polyhedron $X=A(w)$
 \[
 \La(X)=\begin{cases}
 0 & \text{if } w \text{ has no spherical ends},\\
 \mZ & \text{if one end of } w \text{ is spherical},\\
 \De & \text{ if both ends of $w$ are spherical}.
 \end{cases} 
 \]
 Hence, we obtain the following result.
 \begin{corol}\label{46} 
 If $X$ is a band or string polyhedron, then $g(X)=1$, except the case when $X=A(w)$ and both ends of
 the word $w$ are spherical. In the latter case $g(X)=(p-1)/2$.
 \end{corol}
 
 \section{Case $n>4(p-1)$}
 \label{s5} 
 
 For $n=4p-3$ we set $m=6p-5=n+2p-2$ and $q=2(n-1)=n+4p-5=m+2p-3$. Then $\cA$ contains Moore 
 polyhedra $M^q_k$ (including $S^q=M^q_0$) and $\cB$ contains the shifted Chang polyhedron 
 $C^m=C_{00}[2p-2]$. Let $\kN_k=\hos_p(M^q_k,C^m)$. Applying $\hos_p(M^q_k,\_\,)$ to the 
 cofibre sequence
 \[
  0\to S^{m-1}\to S^n\to C^m\to S^m \to S^{n+1}
 \]
 we get an exact sequence
 \[
  0\to \mZ/p \xarr{\la} \kN_k\xarr{\mu} \mZ/p\to 0.
 \]
 Thus $\#(\kN_k)=p^2$. On the other hand, applying $\hos_p(\_\,,C)$ to the cofibre sequence
 \eqref{Ekd} of Section~\ref{s2}, we get an exact sequence
 \[
 \kN_0\xarr{p^k} \kN_0\xarr{\eta} \kN_k\to 0. 
 \]
 Therefore the map $\eta$ is an isomorphism. Setting $k=1$, we see that $p\kN_0=0$,
 hence $\kN_0\simeq\mZ/p\xx\mZ/p$ and $\kN_k\simeq\mZ/p\xx\mZ/p$ for all $k$. 
 We denote by $\la_k$ a generator of $\kN_k$ which is in $\im\la$ and by $\mu_k$ a 
 generator of $\kN_k$ such that $\mu(\mu_k)\ne0$.
 
 Analogous observations show that the generator of the cyclic group $\kM^{qq}_{kl}=\hos_p(M^q_l,M^q_k)$ 
 induces an isomorphism $\kN_k\to \kN_l$ if $k\ge l>0$ and zero map if $0<k<l$. On the other hand,
 the diagram \eqref{e31} implies that an element $(a,b)$ of the ring $\De=\hos_p(C,C)$ acts on $\kN_k$   
 as multiplication by $a$ (recall that $a\equiv b\!\pmod p$). Therefore, a map $\vi:A\to B$, where $A$ is 
 a wedge of Moore polyhedra $M^q_k$ and $B$ is a wedge of Chang polyhedra $C^m$ can be considered as
 a block matrix $\Phi=(\Phi_{ik})_{\substack{k\in\mno\\ i=1,2}}$, where all blocks are with coefficients
 from $\mZ/p$ and both horizontal stripes $\Phi_1,\,\Phi_2$ have the same number of rows. Namely, 
 $\Phi_{1k}$ consists of coefficients at $\la_k$ and $\Phi_{2k}$ consists of coefficients at $\mu_k$.
 Two such matrices define isomorphic objects from $\cE(\kM)$ \iff one of them can be transformed to the
 other by a sequence of the following transformations:
 
 $\Phi_1\mapsto T\Phi_1$ and $\Phi_2\mapsto T\Phi_2$ with the same invertible matrix $T$;
 
 $\Phi^k\mapsto \Phi^kT^k$ for some invertible matrix $T^k$;
 
 $\Phi^k\mapsto \Phi^k+\Phi^lU_{lk}$ for any matrix $U_{lk}$ of the appropriate size, where $l>k$ or
 $l=0<k$.
 \\
 It is well-known that this matrix problem is \emph{wild}, i.e. contains the problem of classification
 of pairs of linear maps in a vector space; hence, a problem of classification of representations of
 any finitely generated algebra over the field $\mZ/p$ (cf. \cite[Section 5]{d0}). Namely, consider
 the case when the matrix $\Phi=\Phi(F,G)$ is of the form 
 \[
 \left[\begin{array}{cc|c}
  I & 0 & 0 \\
  0 & I & 0 \\
  \hline
  F & I & 0 \\
  G & 0 & I
  \end{array} \right]
 \]
 Here $I$ is a unit matrix of some size, $F$ and $G$ are arbitrary square matrices of the same size;
 line show the subdivision of $\Phi$ into blocks $\Phi_{ik}$ (there are only two vertical stripes). 
 One easily checks that $\Phi(F,G)$ and
 $\Phi(F',G')$ define isomorphic objects \iff there is an invertible matrix $T$ such that $F'=TFT^{-1}$
 and $G'=TGT^{-1}$. So we obtain the following result.
 
 \begin{theorem}\label{51} 
 The classification of $p$-local polyhedra in $\cS^n_p$ for $n>4(p-1)$ is a wild problem.
 \end{theorem}

 \end{document}